\documentclass{article}

\usepackage{arxiv}

\usepackage[utf8]{inputenc} 
\usepackage[T1]{fontenc}    
\usepackage{hyperref}       
\usepackage{url}            
\usepackage{booktabs}       
\usepackage{amsfonts}       
\usepackage{nicefrac}       
\usepackage{microtype}      
\usepackage{lipsum}
\usepackage{amsmath}

\usepackage{amsfonts,amssymb}
\usepackage{graphicx}
\usepackage{amsthm}
\usepackage[dvipsnames]{xcolor}

\newcommand{\R}{\mathbb{R}}
\newcommand{\s}{\mathbb{S}}
\newcommand{\e}{\mathbb{E}}

\newtheorem{theorem}{Theorem}[section]

\newtheorem{lemma}{Lemma}[section]
\newtheorem{corollary}{Corollary}[section]

\usepackage{cite}
\usepackage{authblk}

\def\sign{\text{sign}\,}
\def\supp{\text{supp}\,}
\newcommand{\p}{\partial}
\newcommand{\bb}{\begin{equation}}
\newcommand{\ee}{\end{equation}}
\newcommand{\ba}{\begin{array}}
\newcommand{\ea}{\end{array}}
\newcommand{\f}{\frac}
\usepackage[all]{xy}
\newcommand{\ds}{\displaystyle}

\newcommand{\al}{\alpha}

\numberwithin{equation}{section}

\title{Geometrical demonstration for persistence properties for a bi-Hamiltonian shallow water system} 

\author{Igor Leite Freire}

\affil{Centro de Matem\'atica, Computa\c{c}\~ao e Cogni\c c\~ao,\\ Universidade Federal do ABC,\\ Avenida dos Estados, $5001$, Bairro Bangu,\\
$09.210-580$, Santo Andr\'e, SP - Brasil\\
  \texttt{igor.freire@ufabc.edu.br} \\
  \texttt{igor.leite.freire@gmail.com}}

\begin{document}
\maketitle
\begin{abstract}
\centering\begin{minipage}{\dimexpr\paperwidth-9cm}
\textbf{Abstract:} 
We present a geometrical demonstration for persistence properties for a bi-Hamiltonian system modelling waves in a shallow water regime. Both periodic and non-periodic cases are considered and a key ingredient in our approach is one of the Hamiltonians of the system. As a consequence of our developments we improve recent works dealing with unique continuation properties for shallow water equations, as well as we provide a novel way to prove that the unique compactly supported solution of the system is necessarily the zero function.
\vspace{0.1cm}

\textbf{2010 AMS Mathematics Classification numbers}: 35Q35, 35Q51, 37K10.


\end{minipage}
\end{abstract}
\bigskip
\newpage
\tableofcontents
\newpage

\section{Introduction}\label{sec1}

The Camassa-Holm (CH) equation, up to a Galilean boost, is given by
\begin{equation}\label{1.0.1}
u_t-u_{txx}+\al u_x+3uu_x=2u_xu_{xx}+uu_{xxx},
\end{equation}
where $u=u(t,x)$ and $\al$ is a constant. Despite being discovered by Fokas and Fuchssteiner in the early 80's in the context of integrable equations, see Fokas and Fuchssteiner \cite{fuch-fok}, Fuchssteiner \cite{fuch}, and also the comments on \cite[page 146]{fokas}, it was not spotlighted until the seminal work by Camassa and Holm \cite{chprl}, when \eqref{1.0.1} was added as a new member to the zoo of models describing waves in shallow water regimes. Their work brought the equation to light in view of its physical relevance as a shallow water model and the fact that it has quite interesting mathematical  solitonic solutions -- the (multi-)peakons --, to name a few, as well as due to the fact that it is bi-Hamiltonian (a fact known since Fokas and Fuchssteiner's works).

If we drop the right side of \eqref{1.0.1}, then we recover the BBM equation, which was proposed by Benjamin, Bona and Mahony \cite{benjamin} as an alternative to the KdV equation \cite{kdv} for describing waves in water models. However, both KdV and BBM equations do not describe wave breaking, which is a common, but mathematically challenging, phenomenon arising in waves propagating in shallow water regimes.

If we take $\al=0$ in \eqref{1.0.1}, then we can rewrite it in its most common form 
\begin{equation}\label{1.0.2}
m_t+2u_xm+um_x=0,
\end{equation}
where $m:=u-u_{xx}$ is known as momentum \cite{chprl}. The blow up of its solutions, manifested through a singularity in its first order spatial derivative (wave breaking) was proved by Constantin and Escher \cite{const1998-1,const1998-2} and \cite{const1998-3} for the non-periodic and periodic cases, respectively. See also \cite{const2000-1}. Moreover, in \cite{constjmp,henry1} it was shown that the equation does not have non-zero compactly supported solutions for any $t$, except eventually at $t=0$ (corresponding to the initial data).

The CH equation is integrable in several different senses, such as: it is bi-Hamiltonian and possesses a Lax pair \cite{chprl,chh}; its solutions provide metrics for a two-dimensional Riemannian manifold with constant and negative Gaussian curvature \cite{reyes2002,keti2015,raspajde,raspaspam} (pseudo-spherical equations) and is geometrically integrable as well  \cite{reyes2002}. For further details on integrable or pseudo-spherical equations, see \cite{olverbook} or \cite{chern,sa}, respectively.

More recently, Linares and Ponce \cite{linares} proved a unique continuation result for the solutions of the CH equation, which can be summarised as: if there exists an open set in which $u$ is defined and vanishes there, then the solution vanishes everywhere. Shortly after, their ideas were retaken by the author \cite{freire-jpa}, but now with focus on the the Dulling-Gottwald-Holm (DGH) equation (that has the CH equation \eqref{1.0.2} as a particular member). 

The approach in \cite{freire-jpa} used some of the ideas introduced in \cite{linares}, but differently from that reference the conserved quantities of the equation were utilised and brought a different framework in the study of persistence properties of the DGH and CH equations: they were used to give a novel demonstration for unique continuation properties of the solutions of \eqref{1.0.2} and went further, exploring some consequences, such as the non-existence of compactly supported solutions of the CH in Sobolev spaces. These ideas were also employed in the investigation of persistence properties of the BBM equation in \cite{raspa-mo} and the $0-$equation in \cite{raspa-0,raspa-ge}. 

It is also worth mentioning that in \cite{freire-jpa} it was shown that the ideas introduced in \cite{linares} and the use of conserved quantities are not only geometric, but also consistent with the physical derivation of the model considered. 

Since the end of the first decade of our century, some generalisations of the CH equation have been considered, such as systems involving it. Perhaps one of the most relevant and investigated is
\begin{equation}\label{1.0.3}
\left\{\ba{l}
m_t+2u_xm+um_x+\rho\rho_x=0,\\
\\
\rho_t+(u\rho)_x=0,
\ea\right.
\end{equation}
where $\rho=\rho(t,x)$ and again $m=u-u_{xx}$, $u=u(t,x)$; $t$ and $x$ denote time and space, respectively, whereas $u$ is the horizontal velocity of the fluid and $\rho$ is related to the horizontal deviation of the fluid's surface (in the Section \ref{sec5} these variables will be better discussed). This system was firstly reported by Olver and Rosenau, see \cite[Eq. (43)]{olver-pre}, and later derived from the Green-Naghdi equations \cite{g-n} as a shallow water model by Constantin and Ivanov \cite{const-ivanov}. 

System \eqref{1.0.3} is often known as CH2 or CH(2,1) system, see \cite[Sect. 3]{holm-ivanov}, and it reduces to the CH equation \eqref{1.0.2} whenever $\rho=0$. It is interesting to note that the CH2 system and the CH equation were both born in the context of integrable systems in \cite{fuch-fok,fuch} and \cite{olver-pre}, respectively, and latter shown in \cite{chprl} and \cite{const-ivanov} their physical relevance as shallow water models. Their similarities, however, are not restricted to their birth since the CH2 system shares several other properties with the CH equation, such as: 
\begin{itemize}
    \item It is bi-Hamiltonian \cite[Ex. 4, page 53]{olver-pre} and hence, integrable in this sense;
    \item It is also integrable in the sense of having a Lax-pair ({\it e.g}, see \cite[page 7130]{const-ivanov}), that is, it is the compatibility condition of the linear system
$$
\ba{lcl}
\psi_{xx}&=&\ds{\left(\zeta m-\zeta^2\rho^2+\f{1}{4}\right)\psi},\\
\\
\psi_t&=&\ds{\left(\f{1}{2\zeta}-u\right)\psi_x+\f{1}{2}u_x\psi},
\ea
$$
see \cite{holm-ivanov} for further details;
    \item It is also of pseudo-spherical type \cite{chris}, that is, the solutions of the system define one-forms that satisfy certain structure equations that implies on the existence of a two-dimensional Riemannian manifold with constant Gaussian curvature equals to $-1$;
    \item It is locally well-posed provided that $(u(0,\cdot),\rho(0,\cdot))\in H^{s}(\R)\times H^{s-1}(\R)$, for $s\geq2$, see \cite[Theorem 2.1]{henry2} and \cite{const-ivanov}; its solutions, under some conditions, develop singularities in finite time in the form of wave breaking as well as it also possesses global solutions \cite{const-ivanov}. These works are concerned with problems defined in the whole line, but there are also results about local and global existence of solutions as well as blow-up (wave breaking) for the system \eqref{1.0.3} provided that $u(0,\cdot)$ and $\rho(0,\cdot)$ are periodic in $H^{s}(\s)\times H^{s-1}(\s)$, with $s\geq2$, see \cite{liu-ima,hu-proc};
    \item It does not have non-trivial compactly supported solutions defined on the whole real line. In fact, Henry \cite{henry2} proved that if $x\mapsto u(t,x)$ is a $C^2$ function, compactly supported, and $(u,\rho)$ solves \eqref{1.0.3}, then $u$ and $\rho$ are identically zero.
    \item System \eqref{1.0.3} can be seen as a geodesic flow on the tangent bundle of suitable infinite-dimensional Lie group \cite{escher-annali}, as well as the CH equation \cite{const2000-1}.
\end{itemize}

The observations above not only show the interdisciplinary relevance of the system \eqref{1.0.3}, which {\it per se} would be enough to justify its importance, but also catch our attention to several remarkable properties shared with the CH equation and makes us conjecture if they are only the edge of an iceberg of common properties, and therefore, it is natural to ask whether the recent developments made in \cite{freire-jpa,linares} could also be somehow applied to \eqref{1.0.3}.

In line with the last paragraph, some natural questions are very natural:
\begin{itemize}
    \item[{\bf Q1}] Is it possible to prove a unique continuation result similar to those proved in \cite{linares,freire-jpa} for the system \eqref{1.0.3}?
    \item[{\bf Q2}] As a consequence of the unique continuation of solutions proved in \cite{linares,freire-jpa}, in \cite{freire-jpa} was shown that the periodic DGH equation (which includes the CH) does not have non-trivial compactly supported solutions. By a periodic and compactly supported solution we mean a periodic solution such that the support of its restriction to $[0,1)$ is compact. As far as the author knows, the following question has not been addressed so far for the CH2 system: Can the periodic solutions of \eqref{1.0.3} be compactly supported?
    
    \item[{\bf Q3}] Is it possible to conclude that the system \eqref{1.0.3} does not have non-trivial compactly supported solutions defined on suitable Banach spaces of functions defined on $\R$, similarly as proved in \cite{freire-jpa}?
\end{itemize}

The present paper is dedicated to respond these questions. In fact, we report an extension and generalisation of the results in \cite{freire-jpa} (and \cite{linares} as well).

As a consequence of our approach and results for the unique continuation of solutions of the CH2 system, we are not merely able to recover the non-existence results of compactly supported solutions of \eqref{1.0.3} proved by Henry \cite{henry2}. On the contrary, the venue paved gives us tools for investigating non-existence results of periodic compactly supported solutions of the system \eqref{1.0.3} that had not been explored previously. 

Similarly to \cite{freire-jpa}, our strategy to tackle {\bf Q1}, {\bf Q2} and {\bf Q3} is essentially geometric and physically consistent with the assumptions made to derive the system \eqref{1.0.3}. For a certain fixed time $t$ we first construct an auxiliary function that generalises the one introduced in \cite{linares} and use the non-locality of the convolution to extend to a straight line/circle local properties of the solutions. Then we use one of the Hamiltonians to extend to any other time the result proved for the fixed $t$. Our approach can be interpreted as follows, from two different viewpoints:
\begin{itemize}
    \item From a physical framework, system \eqref{1.0.3} describes the propagation of waves on the surface of a fluid. Roughly speaking, a wave is a disturbance, or perturbation, propagating on a media. If we can find a time for which the energy of a conservative system vanishes, then this implies that the phenomenon under consideration did not occur. In our case, since the system \eqref{1.0.3} describes the propagation of waves in a fluid, if we are able to find a time for which the energy of the system disappears, then it means that we did not have an initial perturbation (initial data) and, therefore, did not have the propagation of waves;
    \item From a mathematical point of view, if we can find a straight line segment (for the non-periodic problem) or an arc segment (for the periodic problem) for which the solution does not depend on $x$ or is constant, then we use the non-locality of the convolution to prove that this local behaviour is extended to the whole straight-line or circle containing the original segment/arc. Hence, we use the conservation of energy to show that this also holds to any parallel straight-line/circle of this first one, which therefore extends this property to the entire domain of the solution.
\end{itemize}

{\bf Notation and conventions.} Throughout this paper we identify the interval $[0,1)$ with the circle $\s$. We denote by $\mathbb{E}$ the real line $\R$ (non-periodic problem) or $\mathbb{S}$ (periodic case), while $L^1(\mathbb{E})$ corresponds to the space of the integrable functions on $\mathbb{E}$. The norm of a Banach space $X$ is denoted by $\|\cdot\|_X$. Given $f,\,g\in L^1(\e)$, their convolution is denoted by $f\ast g$, whereas for each $s\in\R$, the Sobolev space of order $s$ is denoted by $H^s(\mathbb{E})$ (for further details, see \cite{linaresbook,taylor}). If $u=u(t,x)$, we use $u_0(x)$ to denote the function $x\mapsto u(0,x)$. Moreover, we also define $m:=u-u_{xx}$, for $t>0$, and $m_0:=u_0-u_0''$. Let $\lfloor \cdot \rfloor$ denote the greatest integer function and we recall the Helmholtz operator $\Lambda^2=1-\p_x^2$ and its inverse $\Lambda^{-2}=(1-\p_x^2)^{-1}$, given by $\Lambda^{-2}(f)=g\ast f$, where 
\begin{align}\label{1.0.4}
g(x)=\left\{\ba{ll}
\ds{\f{e^{-|x|}}{2}},&\text{if}\,\,\e=\R,\\
\\
\ds{\frac{\cosh(x-\lfloor x \rfloor - 1/2)}{2\sinh(1/2)}}, & \text{if}\,\,\e=\mathbb{S}.
\ea
\right.
\end{align}

In the case $\e=\s$, we add the additional periodic condition for the solutions considered herein
\begin{equation}\label{1.0.5}
\ba{lcl}
u(t,x+1)&=&u(t,x), \quad (t,x)\in[0,T)\times\R,\\
\\
\rho(t,x+1)&=&\rho(t,x), \quad (t,x)\in[0,T)\times\R.
\ea
\end{equation}

Finally, very often we use the jargon {\it non-trivial} for functions/solutions. By a non-trivial function/solution, we mean a function/solution that is not identically zero.

{\bf Outline of the paper.} In the next section we present the main theorems of the paper. In Section \ref{sec3} we prove technical lemmas and theorems that will be of vital importance for the demonstrations of our main results, which are proved in Section \ref{sec4} and discussed in Section \ref{sec5}, while in Section \ref{sec6} we present the conclusions of the work.

\section{Main results}\label{sec2}

We begin by noticing that the system \eqref{1.0.3} has the following Hamiltonian \cite{olver-pre}, which gives a conserved quantity for the equation:
\begin{equation}\label{2.0.1}
{\cal H}(t)=\f{1}{2}\int_\e(u^2+u_x^2+\rho^2)dx=\f{1}{2}\|u(t,\cdot)\|^2_{H^1(\e)}+\f{1}{2}\|\rho(t,\cdot)\|^2_{L^2(\e)}.
\end{equation}

Let us enumerate some conditions that we shall use frequently during this section. We firstly assume the existence of real numbers $t_0$, $x_0$, $t_1$ and $x_1$, such that $t_0<t_1$ and $x_0<x_1$, defining a non-empty open set 
\begin{equation}\label{2.0.2}
{\cal R}=(t_0,t_1)\times(x_0,x_1)\subseteq[0,T)\times\mathbb{E},
\end{equation}
for which we shall require certain properties of the solutions of \eqref{1.0.3}. Geometrically speaking, if $\mathbb{E}=\R$ the set ${\cal R}$ is an open rectangle, while it can be identified as a piece of cylinder whenever $\mathbb{E}=\s$. Also, we assume that the solution $(u,\rho)$ is bounded and exists for some $T>0$.

The aforesaid conditions are:
\begin{enumerate}
    \item[{\bf C1}] $(u,\rho)$ is a solution of \eqref{1.0.3} such that
    $$
    \left(u^2+\f{1}{2}u_x^2+\f{1}{2}\rho^2\right)(t,\cdot)\in C^0(\R),\,\,t\in[0,T);
    $$
    \item[{\bf C2}] $(u,\rho)\big|_{\cal R}=(0,0)$, where ${\cal R}$ is given by \eqref{2.0.2};
    \item[{\bf C3}] The solution $(u,\rho)$ conserves \eqref{2.0.1}.
\end{enumerate}

We observe that the condition {\bf C1} is very mild since any solution $(u,\rho)$ of the system \eqref{1.0.3} such that $u(t,\cdot)\in C^1(\e)$ and $\rho(t,\cdot)\in C^0(\e)$ satisfies it. In particular, the solutions satisfying the existence and uniqueness conditions required in \cite{const-ivanov,henry2,liu-ima,hu-proc} fulfill such requirement. Note, however, that we do not request uniqueness of solutions. 

Condition {\bf C3} is a structural property of the system \eqref{1.0.3}. In fact, if we define
$$C^0:=\f{u^2+u_x^2+\rho^2}{2},\quad C^1:=u^3-u^2u_{xx}-uu_{tx}+u\rho^2$$
we have the identity
$$
\p_t\Big(C^0\Big)+\p_x\Big(C^1\Big)=u\Big(m_t+2u_xm+um_x+\rho\rho_x\Big)+\rho\Big( \rho_t+(u\rho)_x\Big),
$$
meaning that on the solutions of \eqref{1.0.3}, the divergence of the current $C:=(C^0,C^1)$ vanishes, that is, it is a conservation law for the equation. As such, the component $C^0$, which is called {\it conserved density}, gives us the invariant \eqref{2.0.1}, called {\it conserved quantity}. In fact, if we define $\p\e=\{-\infty,+\infty\}$, if $\e=\R$, or $\p\e=\{0,1\}$, if $\e=\s$, then
$$
\f{d}{dt}{\cal H}(t)=\f{d}{dt}\int_\e C^0dx=-\int_\e \p_x C^1dx=-C^1\Big|_{\p\e}.
$$

Therefore, as long as $C^1\Big|_{\p\e}=0$, \eqref{2.0.1} is constant along time. In particular, if $u_0\in H^s(\e)$ and $\rho_0\in H^{s-1}(\e)$, $s\geq2$, then not only we have granted the existence and uniqueness of the corresponding solution, see \cite{henry2,hu-proc,liu-ima}, but also we have {\bf C1} and {\bf C3} satisfied.

\begin{theorem}\label{teo2.1}
If $(u,\rho)\in C^0([0,T);H^s(\e)\times H^{s-1}(\e))$, $s\geq2$, is a solution of \eqref{1.0.3} satisfying {\bf C2}. Then necessarily $u=\rho=0$.
\end{theorem}

Theorem \ref{teo2.1} has the following non-existence result as a straightforward consequence:

\begin{corollary}\label{cor2.1}
If $(u_0,\rho_0)\in H^s(\e)\times H^{s-1}(\e)$, $s\geq 2$, is a non-trivial initial data for \eqref{1.0.3}, then there is no open subset of $[0,T)\times\e$ for which the corresponding solution $(u,\rho)$ vanishes.
\end{corollary}

Theorem \ref{teo2.1} says that if the horizontal velocity of the fluid and its horizontal deviation are constant on a segment during an interval of time $(t_0,t_1)$, then it implies that $u=\rho=0$ everywhere. As a result we do not have any perturbation propagating on the fluid's surface. On the other hand, as long as we impose that  \eqref{1.0.3} is subject to an initial data $(u,\rho)(0,x)=(u_0(x),\rho_0(x))\not\equiv0$, then we cannot find any open set on $[0,T)\times\e$ for which the solution vanishes, as stated in Corollary \ref{cor2.1}.

A small digression about compactly supported functions would be welcome at this stage. Let us begin with functions defined on the whole line (non-periodic case). The supports of the functions $u$ and $\rho$ are respectively given by
$$\supp{(u(t,\cdot))}:=\overline{\{u(t,x)\neq0,\,x\in\R\}}$$ 
and
$$\supp{(\rho(t,\cdot))}:=\overline{\{\rho(t,x)\neq0,\,x\in\R\}},$$
respectively. Let us consider the case where $u(t,\cdot)$ is compactly supported. This implies the existence of numbers $a_t$ and $b_t$ such that $\supp{(u(t,\cdot))}=[a_t,b_t]$. Note that as $t$ varies, the numbers $a_t$ and $b_t$ may also change.

The paragraph above is enough to clarify the situation for the non-periodic case. Let us now catch sight of the periodic case. Firstly, we note that if a function is non-trivial, then it cannot be compactly supported in the sense defined above because of the condition \eqref{1.0.5}. On the other hand, a solution of \eqref{1.0.3} satisfying \eqref{1.0.5} is completely known if we restrict it on $[0,1)$. Then, by a periodic compactly supported solution of \eqref{1.0.3} we mean a solution such that
$$
\supp{(u\big|_{[0,1)}(t,\cdot))}:=\overline{\{u(t,x)\neq0,\,x\in[0,1)\}}
$$ 
and 
$$\supp{(\rho\big|_{[0,1)}(t,\cdot))}:=\overline{\{\rho(t,x)\neq0,\,x\in[0,1)\}}$$
is compact.

We close this section with the following characterisation of the solutions of the CH2 system.
\begin{theorem}\label{teo2.2}
Suppose that $(u,\rho)$ is a solution of the CH2 system \eqref{1.0.3} defined on $[0,T)\times\e$, for some $T>0$, satisfying {\bf C3}. If there exists $t^\ast\in(0,T)$ and an open set $(x_0,x_1)\subseteq\e$ such that $u(t^\ast,x)=\rho(t^\ast,x)=0$, $x\in(x_0,x_1)$, then the solution $(u,\rho)$ vanishes everywhere. 
\end{theorem}

Note that we do not prevent the situation $\rho=0$ in the theorem above. Therefore, if we add this hypothesis, then the result is immediately concerned with the CH equation. We, however, avoid such an explicit statement because it is straightforward. Still, observe that Theorem \ref{teo2.2} implies on the non-existence of non-trivial compactly supported solutions of the CH2 system and the CH equation as well.

\begin{corollary}\label{cor2.2}
A non-trivial solution $(u,\rho)\in H^s(\e)\times H^{s-1}(\e)$, $s\geq2$, of \eqref{1.0.3} satisfying {\bf C1} and {\bf C3} cannot have the functions $u$ and $\rho$ simultaneously compactly supported for any value of $t>0$.
\end{corollary}

\section{Preliminary results}\label{sec3}

We begin with noticing that \eqref{1.0.3} can be rewritten in the non-local form
\begin{equation}\label{3.0.1}
\left\{\ba{l}
\ds{u_t+uu_x=-\p_x\Lambda^{-2}\left(u^2+\f{1}{2}u_x^2+\f{1}{2}\rho^2\right)},\\
\\
\ds{\rho_t+(u\rho)_x=0.}
\ea\right.
\end{equation}

Throughout this section we shall assume that the solution $(u,\rho)$ of the system \eqref{1.0.3} exists for $t\in[0,T)$, for some $T>0$.

\begin{lemma}\label{lema3.1}
Let $(u,\rho)$ be a solution of \eqref{1.0.3} satisfying ${\bf C3}$. If there exists $t_0\in\R$ such that ${\cal H}(t_0)=0$, then ${\cal H}(t)=0$, for all $t$. In particular, $u(t,x)=\rho(t,x)=0$ for all $[0,T)\times\e$.
\end{lemma}

\begin{proof}
Since ${\cal H}(t)$ is conserved and ${\cal H}(t_0)=0$, then ${\cal H}(t)={\cal H}(t_0)=0$, $t\in[0,T)$. Moreover, we have $0=2{\cal H}(t)=\|u(t,\cdot)\|^2_{H^1(\e)}+\|\rho(t,\cdot)\|_{L^2(\e)}^2$, which implies $(u,\rho)=(0,0)$. 
\end{proof}

Lemma \ref{lema3.1} is nothing but the conservation of energy of the physical phenomenon described by \eqref{1.0.3}. Its essence is: if the energy of the system vanishes at some specific time $t_0$, then it vanishes everywhere.

\begin{lemma}\label{lema3.2}
Let $(u,\rho)$ be a solution of the system \eqref{1.0.3} satisfying the condition ${\bf C1}$ and let us define, for each $t\in[0,T)$, the real valued function
\begin{equation}\label{3.0.2}
   x\mapsto  f_t(x):=\left(u^2+\f{1}{2}u_x^2+\f{1}{2}\rho^2\right)(t,x).
\end{equation}
Then:
\begin{enumerate}
    \item For each $t\in[0,T)$, the function $f_t(\cdot)$ is non-negative and continuous.
    \item If the solution satisfies the condition ${\bf C3}$ and there exists $t^\ast\in[0,T)$ such that $f_{t^\ast}(x)=0$, for each $x\in\e$, then $(u,\rho)\equiv(0,0)$. In particular, $f_t(\cdot)\equiv0$, for all $t\in[0,T)$.
\end{enumerate}
\end{lemma}

\begin{proof}
It is immediate that $f_t(\cdot)$ is non-negative and continuous by construction. Let us assume the existence of $t^\ast$ such that $f_{t^\ast}(x)=0$, for all $x\in\e$. Then we conclude that $u(t^\ast,\cdot)^2=\rho(t^\ast,\cdot)^2=0$ and the result is a consequence of Lemma \ref{lema3.1}.
\end{proof}

Let us interpret the function defined in \eqref{3.0.2} for solutions conserving \eqref{2.0.1}. The following identity is immediate from \eqref{2.0.1} and \eqref{3.0.2}:
\begin{equation}\label{3.0.3}
{\cal H}(t)+\f{1}{2}\|u(t,\cdot)\|^2_{L^2(\e)}=\int_\e f_t(x)dx\Rightarrow 0\leq{\cal H}(t)\leq \int_\e f_t(x)dx\leq 2{\cal H}(t).
\end{equation}

From \eqref{3.0.2} and \eqref{3.0.3} we conclude that ${\cal H}(t)$ vanishes if and only if $f_t(\cdot)$ does vanish too. In particular, if this happens, then $u(t,\cdot)\equiv0$.
\begin{lemma}\label{lema3.3}
Let $(f_t)_{t\in[0,t)}$ be the family of functions given by \eqref{3.0.2}. For each $t\in[0,T)$, $\Lambda^{-2}f_{t}(x)=0$ if and only if $f_t(x)=0$.
\end{lemma}
\begin{proof}
We recall that
$$
\Lambda^{-2}f_{t}(x)=\int_\e g(x-y)f_t(y)dy.
$$

The result is immediate due to the following observations: firstly, the function $g(\cdot)$ given by \eqref{1.0.4} is strictly positive. Secondly, the function $f_t(x)$ is non-negative, for each $t\in[0,T)$. Finally, from the previous observations, the convolution $g\ast f_t=0$ if and only if $f_t=0$.
\end{proof}

Note that if we are able to find one, and only one, number $t^\ast$ for which $f_{t^\ast}(x)=0$, then \eqref{3.0.3} jointly with Lemma \ref{lema3.1} imply that the solution of the system \eqref{1.0.3} vanishes everywhere. 

Let us now define a second family of functions, given by $F_t(x)=(\p_x\Lambda^{-2}f_t)(x)$, that is
\begin{equation}\label{3.0.4}
F_t(x)=\int_\e \p_xg(x-y)\left(u^2+\f{1}{2}u_x^2+\f{1}{2}\rho^2\right)(t,y)dy.
\end{equation}

\begin{theorem}\label{teo3.1}
Let $(F_t(\cdot))_{t\in[0,T)}$ and $(f_t(\cdot))_{t\in[0,T)}$ be the families of functions given by \eqref{3.0.4} and \eqref{3.0.2}, respectively. If $(u,\rho)$ satisfies the conditions ${\bf C1}$ and ${\bf C3}$, and there exists numbers $t^\ast$, $x_0$ and $x_1$, with $x_0<x_1$, such that $\{t^\ast\}\times[x_0,x_1]\subseteq (0,T)\times\e$, $f_{t^\ast}\big|_{(x_0,x_1)}\equiv0$ and $F_{t^\ast}(x_0)=F_{t^\ast}(x_1)$, then $F_t(\cdot)\equiv0$, for all $t\in[0,T)$. In particular, $f_t(\cdot)\equiv0$ and $(u,\rho)(t,\cdot)\equiv0$, for all $t\in[0,T)$. 
\end{theorem}

\begin{proof}
Since $F_t(\cdot)\in C^1(\R)$ and recalling that $\p_x^2\Lambda^{-2}=\Lambda^{-2}-1$, if $x\in(x_0,x_1)$, then
$$F_{t^\ast}'(x)=(\Lambda^{-2}f_{t^\ast})(x)-f_{t^\ast}(x)=(\Lambda^{-2}f_{t^\ast})(x).$$
By the Fundamental Theorem of Calculus we have
$$
0=F_{t^\ast}(x_1)-F_{t^\ast}(x_0)=\int_{x_0}^{x_1}(\Lambda^{-2}f_{t^\ast})(y)dy.
$$
Since $(\Lambda^{-2}f_{t^\ast})(y)\geq0$, we necessarily have $(\Lambda^{-2}f_{t^\ast})(y)=0$. The result is then a consequence of lemmas \ref{lema3.3} and \ref{lema3.2}.
\end{proof}

Let us explore \eqref{3.0.4} a little more. We note that
$$
\p_x g(x)=\left\{
\begin{array}{lcl}
\ds{-\f{1}{2}\sign{(x)}e^{-|x|}},&\text{if}&\e=\R,\\
\\
\ds{\frac{\sinh(x-\lfloor x \rfloor - 1/2)}{2\sinh(1/2)}},&\text{if}&\e=\s\,\,\text{and}\,\,x\neq0.
\end{array}
\right.
$$

For convenience, in the next lemma we denote $\p_xg(\cdot)$ by $g'(\cdot)$.

\begin{lemma}\label{lema3.4}
Suppose that $a,b\in\e$ are numbers such that $a<b$, and let us define $S_{a,b}:\e\rightarrow\e$ by $$
S_{a,b}(y)=\left\{
\begin{array}{lcl}
\ds{\f{1}{2}\sign{(a-y)}e^{-|a-y|}-\f{1}{2}\sign{(b-y)}e^{-|b-y|}},&\text{if}&\e=\R,\\
\\
\ds{\frac{\sinh(b-y-\lfloor b-y \rfloor - 1/2)}{2\sinh(1/2)}-\frac{\sinh(a-y-\lfloor a-y \rfloor - 1/2)}{2\sinh(1/2)}},&\text{if}&\e=\s.
\end{array}
\right.
$$

Then $S_{a,b}\in L^1(\e)$, $S_{a,b}(y)>0$, for all $y\in\e\setminus[a,b]$, and 
\begin{equation}\label{3.0.5}
F_{t}(b)-F_t(a)=\int_{\e}S_{a,b}(y)f_t(y)dy.
\end{equation}

\end{lemma}
\begin{proof}
Firstly, we note that whenever $g'(\cdot)$ is defined, then $S_{a,b}(y)=g'(b-y)-g'(a-y)$.

It follows immediately from the definition of $S_{a,b}$ that it belongs to $ L^1(\e)$, whereas \eqref{3.0.5} is a consequence of \eqref{3.0.4}, the fact that except for a set of measure $0$, $S_{a,b}(y)=g'(b-y)-g'(a-y)$, and the definition of $g'(\cdot)$.

Let us now prove that $S_{a,b}$ is strictly positive outside $[a,b]$. We divide our demonstration in two cases: $\e=\R$ and $\e=\s$.

{\bf Case $\e=\R$.} If $y<a$, then $|a-y|<|b-y|$, $-|a-y|>-|b-y|$ and $e^{-|a-y|}>e^{-|b-y|}$. Moreover, $\sign{(a-y)}=\sign{(b-y)}=+1$ and, therefore,
$$\sign{(a-y)}e^{-|a-y|}>\sign{(b-y)}e^{-|b-y|},$$
which yields the result.

The only other possibility is $y>b$, which implies that $|b-y|<|a-y|$, $\sign{(y-b)}=\sign{(y-a)}=+1$. As a consequence of these observations, we have $e^{-|a-y|}<e^{-|b-y|}$ and $$\sign{(y-a)}e^{-|a-y|}<\sign{(y-b)e^{-|b-y|}},$$
which again implies the result.

{\bf Case $\e=\s$.} We recall that $0\leq a<b<1$. If $y\in[0,a)$ (if $a=0$, then we have nothing to consider), then $0< a-y<b-y<1$, which gives $\lfloor a-y \rfloor=\lfloor b-y \rfloor =0$, and we get the inequality
\begin{equation}\label{3.0.6}
a-y-\lfloor a-y \rfloor -\f{1}{2}<b-y-\lfloor b-y \rfloor -\f{1}{2}.
\end{equation}

On the other hand, if $y\in(b,1)$, then both $b-y$ and $a-y$ belong to $(-1,0)$, meaning that $\lfloor b-y \rfloor=\lfloor a-y \rfloor=-1$ and $a-y<b-y$. Also, we have
\begin{equation}\label{3.0.7}
a-y-\lfloor a-y \rfloor -\f{1}{2}<b-y-\lfloor b-y \rfloor -\f{1}{2}.
\end{equation}

The result follows from \eqref{3.0.6}, \eqref{3.0.7} and the fact that the function $x\mapsto\sinh{(x)}$ is strictly increasing.
\end{proof}

A demonstration that the function $S_{a,b}$ above satisfies the conditions in Lemma \ref{lema3.4} can be inferred from some results proved in \cite[Sect. 2]{linares}, see also \cite[Prop. 4.2]{freire-jpa}. We opted to present a demonstration of this result for sake of completeness and coherence of the present work.

We present a different, but useful, version of Theorem \ref{teo3.1} using Lemma \ref{lema3.4}.

\begin{theorem}\label{teo3.2}
If there are distinct points $a,b\in\e$ such that $F_{t}(a)=F_t(b)$ and $f_t(x)=0$, $x\in(a,b)$, for some $t\in(0,T)$, then $f_{t}(x)=0$, for all $x\in\e.$
\end{theorem}
\begin{proof}
Without loss of generality, we may assume that $a<b$. We firstly observe that $f_t(y)=0$ for $y\in(a,b)$. Secondly, $S_{a,b}(y)$ is positive whereas $f_t(y)$ is non-negative, for all $y\in\e\setminus[a,b]$. Finally, by \eqref{3.0.5} we have
$$
\int_\e S_{a,b}(y)f_t(y)dy=0,
$$
which necessarily implies the result.
\end{proof}

\section{Proof of the main results}\label{sec4}

In the previous section we offered a menu of different ingredients that are now carefully selected to serve to the reader the demonstration of theorems \ref{teo2.1} and \ref{teo2.2} and their corollaries. The price paid for having established earlier these technical results is that now the proofs of our theorems are considerably easier, simpler, shorter and aesthetically more elegant than they would be had we opted to present a direct demonstration of each result.

{\bf Proof of Theorem \ref{teo2.1}.}  For each $t>0$, from \eqref{3.0.1} and \eqref{3.0.4} we can rewrite
\begin{equation}\label{4.0.1}
F_{t}(x)=-(u_t+uu_x).
\end{equation}

Assuming ${\bf C2}$, we conclude that for all $(t,x)\in{\cal R}$, we have $F_t(x)=0$. Fixing $t=t^\ast$, for some $t^\ast\in(t_0,t_1)$, then $F_{t^\ast}(x)=0$, $x\in(x_0,x_1)$ and the result is a consequence of Theorem \ref{teo3.2} and Lemma \ref{lema3.2}.
\hfill$\square$

Corollary \ref{cor2.1}'s demonstration is immediate and for this reason is omitted. Let us move to the second corollary.

{\bf Proof of Theorem \ref{teo2.2}.} If $u(t^\ast,x)=\rho(t^\ast,x)=0$, then the function $f_{t^\ast}(x)=0$, $x\in(a,b)$. On the other hand, evaluating \eqref{1.0.3} at $(t^\ast,x)$ we obtain $(1-\p_x^2)u_t(t^\ast,x)=0$, which implies that $u_t(t^\ast,x)=0$, $x\in[a,b]$. From \eqref{4.0.1} we conclude that $F_{t^\ast}(x)=0$, $x\in[a,b]$. The result is then a consequence of Theorem \ref{teo3.2}. \hfill$\square$

{\bf Proof of Corollary \ref{cor2.2}.}  If $(u,\rho)$ were simultaneously compactly supported, we would then be able to find $t^\ast$ and numbers $a,\,b$, with $a<b$, such that $(u,\rho)(t^\ast,x)=0$, for $x\in[a,b]$, which contradicts Theorem \ref{teo2.2}.
\hfill$\square$

\section{Discussion}\label{sec5}

In this paper a unique continuation result for the solutions of the CH2 system is presented, namely, the one given by Theorem \ref{teo2.1}. As a consequence of our results, we were able to investigate and characterise the existence of compactly supported solutions for such system, as shown by Theorem \ref{teo2.2} or Corollary \ref{cor2.2}. 

It was shown by Henry \cite{henry2} that the CH2 system cannot have non-trivial solutions $(u,\rho)$ if $u$ is compactly supported. Therefore, in regard to this particular fact the results proved in the present paper are not new. Its novelty, however, relies upon the tools developed to conclude the same fact, which are fresh: we showed that the non-existence of compactly supported solutions is an immediate consequence of the unique continuation properties reported in Section \ref{sec2} and proved in Section \ref{sec4}.

As far as the author knows, the paper by Linares and Ponce \cite{linares} was the first to establish a unique continuation result for the solutions of the CH equation. Soon after, the author of the present work applied their ideas to the DGH equation \cite{freire-jpa} and went further, connecting them with conserved quantities. These two ingredients not only gave a demonstration for unique continuation of the DGH equation, but also provided a machinery to investigate the non-existence of compactly supported solutions of the DGH equation and, as a particular case, of the CH equation. 

In line with the precedent paragraph, more recently the ideas in \cite{freire-jpa} were also used for exploring persistence properties of the BBM equation \cite{raspa-mo} and the $0-$equation \cite{raspa-ge,raspa-0}.

Since we recover the results in \cite{linares} for the CH equation when $\rho=0$, let us assume this and explain how our results differ from \cite{linares}. If $\rho=0$ and we fix a value of $t$, then the conditions mentioned in the last paragraph would necessarily imply that the function \eqref{3.0.4} vanishes provided that $(t,x)$ belongs to such open set. As a result, the relation \eqref{3.0.5} tells us that $f_t(x)$ vanishes and for the same value of $t$, the function $u(t,x)=0$, $x\in\e$. The point is: what about other values of $t$? Our approach uses the conservation of energy of the solutions to extend such result for other values of $t$. For a fixed $t$, the inequality \eqref{3.0.3}  implies that the energy of the solutions vanishes if and only if $f_t(\cdot)$ vanishes. If this happens, the invariance of the functional \eqref{2.0.1} implies that the solution is zero for each $t$ as long as the solution exists.

It is also worth mentioning that \eqref{1.0.3} has a ``twin'' system, given by
\begin{equation}\label{5.0.1}
\left\{\ba{l}
m_t+2u_xm+um_x-\overline{\rho}\,\overline{\rho}_x=0,\\
\\
\rho_t+(u\overline{\rho})_x=0,
\ea\right.
\end{equation}
which is also physically relevant \cite{holm-ivanov,const-ivanov}. The $-$ sign in the first equation in \eqref{5.0.1} brings considerable difficulty to investigate the persistence properties of \eqref{5.0.1} in the framework used in this paper because the non-local form of the the first equation in \eqref{5.0.1} is
\begin{equation}\label{5.0.2}
u_t+uu_x=-\p_x\Lambda^{-2}\left(u^2+\f{1}{2}u_x^2-\f{1}{2}\overline{\rho}^2\right)
\end{equation}
which makes us unable to use the ideas from lemmas \ref{lema3.2} and \ref{lema3.3} or theorems \ref{teo3.1} and \ref{teo3.2}.

In spite of the problem mentioned above, we note that if $(u,\rho)$ is a solution of \eqref{1.0.3}, defining $\overline{\rho}=i\rho$, where $i^2=-1$, we transform \eqref{5.0.1} into \eqref{1.0.3}, and vice-versa. Noticing that $\overline{\rho}=0$ if and only if $\rho=0$, we can establish similar results to \eqref{5.0.1} from those proved for \eqref{1.0.3}.

Finally, we would like to observe that most of our analysis used the Hamiltonian 
$${\cal H}(t)=\f{1}{2}\|u\|_{H^1(\e)}^2+\f{1}{2}\|\rho\|_{L^2(\e)}^2,$$
which was the Hamiltonian originally reported by Olver and Rosenau, see \cite[Eq. (42)]{olver-pre} (note that there is a typo in the right side of the expression for $\hat{H}_1$. Our variable $\rho$ corresponds to $v$ in that expression). A more physically consistent choice for the Hamiltonian would be (see \cite{const-ivanov,holm-ivanov})
$${\cal H}(t)=\f{1}{2}\|u\|_{H^1(\e)}^2+\f{1}{2}\|\rho-1\|_{L^2(\e)}^2.$$

Let us explain why the latter is more physically consistent than the former. According to Constantin and Ivanov's derivation of \eqref{1.0.3}, the relation between $\rho$, $u$ and the horizontal deviation of the fluid's surface $\eta$ is (see \cite[page 7130]{const-ivanov})
$$
\rho-1=\f{1}{2}\epsilon\eta-\f{1}{8}\epsilon^2(u^2+\eta^2),
$$
where $\epsilon$ is the ratio between the typical amplitude of the wave and the mean level of the water. In the shallow water regime, we have $\epsilon\ll1$. 

As long as the horizontal deviation of the fluid's surface and the horizontal velocity of the fluid $u$ vanish, then $\rho\rightarrow1$. Therefore, if we want to interpret the result of the present paper from a purely physical point of view, we should replace $\rho$ by $\rho-1$ in our results and any conclusion we have for $\rho$ should be physically concerned with $\rho-1$, like the result for local well posedness for \eqref{1.0.3} reported in \cite{const-ivanov}, which assumed that $\rho-1\in H^2(\R)$ in place of $\rho$, as we did throughout the paper. 

Last, but not least, the demonstrations of the results reported in the present work are essentially geometric, as mentioned in the Introduction, and strongly based on physical arguments, as pointed out throughout the work and emphasised in sections \ref{sec2} and \ref{sec3}.

\section{Conclusion}\label{sec6}

In this paper we proved unique continuation results for the system \eqref{1.0.3}. The ideas introduced here can also be applied to the CH (and other non-local shallow water models), and they not only provide new venues in the subject {\it per se}, but also extend some recent results in the area. As a consequence of this new framework we were able to give a new approach for investigating compactly supported solutions for the system \eqref{1.0.3} in both periodic and non-periodic cases.

\section*{Acknowledgements}

CNPq is also thanked for financial support, through grant nº 404912/2016-8.

\end{document}